\documentclass[12pt]{article}
\usepackage{e-jc}

\usepackage{amsmath,amssymb,latexsym,color,epsfig,enumerate,a4}

\newif\ifnotesw\noteswtrue

\date{}

\newtheorem{conj}{Conjecture}
\newtheorem{cor}{Corollary}
\newtheorem{thm}{Theorem}
\newtheorem{lem}{Lemma}

\newenvironment{proof}{\paragraph{Proof:}}{\hfill$\square$}

\title{On a Conjecture of Thomassen}

\author{Michelle Delcourt\thanks{Research supported by NSF Graduate Research Fellowship DGE 1144245.}\\
\small Department of Mathematics\\[-0.8ex]
\small University of Illinois\\[-0.8ex]
\small Urbana, Illinois 61801, U.S.A.\\
\small\tt delcour2@illinois.edu\\
\and
Asaf Ferber\\
\small Department of Mathematics \\[-0.8ex]
\small Yale University, and \\
\small Department of Mathematics, \\[-0.8ex]
\small MIT. \\
\small \tt asaf.ferber@yale.edu, and ferbera@mit.edu.
}

\date{\dateline{Oct 17, 2014}{XX} }

\begin{document}
\maketitle

\begin{abstract}
\noindent In 1989, Thomassen asked whether there is an
integer-valued function $f(k)$ such that every $f(k)$-connected
graph admits a spanning, bipartite $k$-connected subgraph. In this
paper we take a first, humble approach, showing the
conjecture is true up to a $\log n$ factor.
\end{abstract}
\section{Introduction}
Erd\H{o}s noticed \cite{T88} that any graph $G$ with minimum degree
$\delta(G)$ at least $2k-1$ contains a spanning, bipartite subgraph
$H$ with $\delta(H)\geq k$.  The proof for this fact is obtained by
taking a maximal edge-cut, a partition of $V(G)$ into two sets $A$ and $B$, such that
the number of edges with one endpoint in $A$ and one in $B$, denoted
$|E(A,B)|$, is maximal. Observe that if some vertex $v$ in $A$ does not
have degree at least $k$ in $G[B]$, then by moving $v$ to $B$, one
would increase $|E(A,B)|$, contrary to maximality.  The same argument
holds for vertices in $B$.  In fact this
proves that for each vertex $v\in V(G)$, by taking such a subgraph $H$, the
degree of $v$ in $H$, denoted $d_H(v)$, is at least $d_G(v)/2$. This
will be used throughout the paper.

Thomassen observed that the same proof shows the following stronger
statement. Given a graph $G$ which is at least $(2k-1)$
\emph{edge-connected} (that is one must remove at least $2k-1$
edges in order to disconnect the graph), then $G$ contains a
bipartite subgraph $H$ for which $H$ is $k$ edge-connected.  In fact,
each edge-cut keeps at least half of its edges. This observation
led Thomassen to conjecture that a similar phenomena also holds for
\emph{vertex-connectivity}.

Before proceeding to the statement of Thomassen's conjecture, we remind the reader that a
graph $G$ is said to be $k$ \emph{vertex-connected} or
$k$-\emph{connected} if one must remove at least $k$ vertices
from $V(G)$ in order to disconnect the graph (or to remain with one
single vertex). We also let $\kappa(G)$ denote the minimum integer
$k$ for which $G$ is $k$-connected. Roughly speaking, Thomassen
conjectured that any graph with high enough connectivity also should
contain a $k$-connected spanning, bipartite subgraph. The following appears
as Conjecture 7 in \cite{T89}.
\begin{conj}\label{conj7}
For all $k$, there exists a function $f(k)$ such that for all graphs $G$, if $\kappa(G) \geq f(k)$,
then there exists a spanning, bipartite $H \subseteq G$ such that $\kappa(H) \geq k$.
\end{conj}

In this paper we prove that Conjecture \ref{conj7} holds up to a $\log n$
factor by showing the following:
\begin{thm}\label{thm:graph}
For all $k$ and $n$, and for every graph $G$ on $n$ vertices the
following holds. If $\kappa(G) \geq 10^{10}k^3 \log n$, then there
exists a spanning, bipartite subgraph $H \subseteq G$ such that
$\kappa(H) \geq k$.
\end{thm}

Because of the $\log n$ factor, we did not try to optimize the
dependency on $k$ in Theorem \ref{thm:graph}.  However, it looks
like our proof could be modified to give slightly better bounds.


\section{Preliminary Tools}
In this section, we introduce a number of preliminary results.

\subsection{Mader's Theorem}

The first tool is the following useful theorem due to Mader \cite{Mad}.

\begin{thm}\label{mader}
Every graph of average degree at least $4\ell$ has an
$\ell$-connected subgraph.
\end{thm}


Because we are interested in finding bipartite subgraphs with high connectivity,
the following corollary will be helpful.

\begin{cor}\label{cor:maderBipartite}
Every graph $G$ with average degree at least $8\ell$ contains a (not
necessarily spanning) bipartite subgraph $H$ which is at least
$\ell$-connected.
\end{cor}

\begin{proof}
  Let $G$ be such a graph and let $V(G)=A\cup B$ be a partition of
  $V(G)$ such that $|E(A,B)|$ is maximal. Observe that $|E(A,B)|\geq
  |E(G)|/2$, and therefore, the bipartite graph $G'$ with parts
  $A$ and $B$ has average degree at least $4\ell$. Now, by applying
  Theorem \ref{mader} to $G'$ we obtain the desired subgraph $H$.
\end{proof}

\subsection{Merging $k$-connected Graphs}
We will also make use of the following easy expansion lemma.

\begin{lem}
  \label{lemma:joining two k connected graphs}
  Let $H_1$ and $H_2$ be two vertex-disjoint graphs, each of which is $k$-connected. Let $H$ be a
  graph obtained by adding $k$ independent edges between these two graphs. Then, $\kappa(H)\geq
  k$.
\end{lem}

\begin{proof}
Note first that by construction, one cannot remove all the edges between $H_1$ and
$H_2$ by deleting fewer than $k$ vertices. Moreover, because $H_1$
and $H_2$ are both $k$-connected, each will remain connected
after deleting less than $k$ vertices. From here, the proof follows easily.
\end{proof}

Next we will show how to merge a collection of a few $k$-connected
components and single vertices into one $k$-connected component.
Before stating the next lemma formally, we will need to introduce
some notation. Let $G_1,\ldots, G_t$ be $t$ vertex-disjoint
$k$-connected graphs, let $U=\{u_{t+1},\ldots,u_{t+s}\}$ be a set
consisting of $s$ vertices which are disjoint to $V(G_i)$ for $1
\leq i \leq t$, and let $R$ be a $k$-connected graph on the vertex
set $\{1,\ldots,t+s\}$. Also let $X = (G_1, \ldots G_t, u_{t+1},
\ldots, u_{t+s} )$ be a $(t+s)$-tuple and $X_i$ denote the $i$th
element of $X$. Finally, let
$\mathcal F_R:=\mathcal F_R(X)$ denote the family
consisting of all graphs $G$ which satisfy the following:

\begin{enumerate}[$(i)$]
\item the disjoint union of the elements of $X$ is a spanning subgraph of $G$, and
\item for every distinct $i, j \in V(R)$ if $ij\in E(R)$, then there exists an
edge in $G$ between $X_i$ and $X_j$, and
\item for every $1 \leq i\leq t$, there is a set of $k$ independent edges between $V(G_i)$ and $k$ distinct vertex sets $\{ V(X_{j_1}), \ldots, V(X_{j_k}) \}$, where $V(u_i)=\{u_i\}$.
\end{enumerate}

\begin{lem}\label{merge few into on}
Let $G_1,\ldots, G_t$ be $t$ vertex-disjoint graphs, each of which
is $k$-connected, and let $U=\{u_{t+1},\ldots,u_{t+s}\}$ be a set of
$s$ vertices for which $U\cap V(G_i)=\emptyset$ for every $1 \leq
i\leq t$. Let $R$ be a $k$-connected graph on the vertex-set
$\{1,\ldots,t+s\}$, and let $X = \{G_1, \ldots G_t, u_{t+1}, \ldots,
u_{t+s} \}$. Then, any graph $G\in \mathcal F_R(X)$ is
$k$-connected.
\end{lem}

\begin{proof}
Let $G\in \mathcal F_R(X)$, and let $S\subseteq V(G)$ be a subset of
size at most $k-1$. We wish to show that the graph $G':=G\setminus
S$ is still connected. Let $x,y\in V(G')$ be two distinct
vertices in $G'$; we show that there exists a path in $G'$
connecting $x$ to $y$. Towards this end, we first note that if both $x$
and $y$ are in the same $G_i$, then because each $G_i$ is
$k$-connected, there is nothing to prove. Moreover, if both
$x$ and $y$ are in distinct elements of $X$ which are also disjoint from $S$, then
we are also finished, as follows. Because $R$ is $k$-connected,
if we delete all of the vertices in $R$ corresponding to elements of $X$ which
intersect $S$, the resulting graph is still connected.
Therefore, one can easily find a path between the elements containing $x$
and $y$ which goes only through ``untouched'' elements of $X$, and hence,
there exists a path connecting $x$ and $y$.

The remaining case to deal with is when $x$ and $y$ are in different
elements of $X$, and at least one of them is not disjoint with $S$. Assume
$x$ is in some such $X_i$ ($y$ will be treated similarly). Using
Property $(iii)$ of $\mathcal F_R$, there is at least one edge
between $X_i$ and an untouched $X_j$.  Therefore one can find a path between $x$ and some vertex $x'$ in an
untouched $X_j$. This takes us back to the previous case.
\end{proof}

\subsection{Main Technical Lemma}

A \emph{directed graph} or \emph{digraph} is a set of vertices and a
collection of directed edges; note that bidirectional edges are
allowed. For a directed graph $D$ and a vertex $v\in V(D)$ we let
$d^+_D(v)$ denote the out-degree of $v$. We let $U(D)$ denote the
\emph{underlying graph of} $D$, that is the graph obtained by
ignoring the directions in $D$ and merging multiple edges.
In order to find the desired spanning, bipartite
$k$-connected subgraph in Theorem \ref{thm:graph}, we look at
sub-digraphs in an auxiliary digraph.

The following is our main
technical lemma and is the main reason why we have a $\log n$
factor.

\begin{lem}\label{lem:main}
If $D$ is a finite digraph on at most $n$ vertices with minimum out-degree
$$\delta^{+}(D) > (k-1) \left\lceil \log n\right\rceil,$$
then there exists a sub-digraph $D' \subseteq D$ such that
\begin{enumerate}
\item For all $v \in V(D')$ we have $d_{D'}^{+}(v) \geq d_{D}^{+}(v) - (k-1) \left\lceil \log n\right\rceil$, and
\item $\kappa(U(D')) \geq k$.
\end{enumerate}
\end{lem}
\begin{proof}
If $\kappa(U(D))\geq k$, then there clearly is nothing to prove. So we may assume that
$\kappa(U(D))\leq k-1$.  Delete a separating set of size at most
$k-1$. The smallest component, say $C_1$, has size at most $n/2$ and
for any $v \in V(C_1)$, every out-neighbor of $v$ is either in $V(C_1)$ or in the
separating set that we removed, and so
$$d_{C_1}^{+}(v) \geq d_{D}^{+}(v) - (k-1).$$

We continue by repeatedly applying this step, and note that this process
must terminate. Otherwise, after at most $\log n$ steps we are left
with a component which consists of one single vertex and yet
contains at least one edge, a contradiction.
\end{proof}

%

\section{Highly Connected Graphs}
With the preliminaries out of the way, we are now ready to prove
Theorem \ref{thm:graph}.
\begin{proof}
Let $G$ be a finite graph on $n$ vertices with $$\kappa(G) \geq
10^{10}k^3 \log n.$$  In order to find the desired subgraph, we first
initiate $G_1:=G$ and start the following process.

As long as $G_i$
contains a bipartite subgraph which is at least $k$-connected on at
least $10^3 k^2 \log n$ vertices, let $H_{i}=(S_i\cup T_i,E_i)$ be
such a subgraph of maximum size, and let $G_{i+1}:=G_i\setminus
V(H_i)$.  Note that $H_1$ must exist as
$$\delta(G_1)\geq 10^{10}k^3\log n-2k\geq 8000k^2\log n,$$
and so by Corollary \ref{cor:maderBipartite},
$G_1$ must contain a $k$-connected bipartite subgraph of size at least $10^3k^2\log n$.

Let $H_1, \ldots, H_t$ be the sequence obtained in this manner, and
note that all the $H_i$'s are vertex disjoint with $\kappa(H_i) \geq
k$ and $|V(H_i)| \geq 10^3 k^2 \log n$. Observe that if $H_1$ is
spanning, then there is nothing to prove. Therefore, suppose for a
contradiction that $H_1$ is not spanning. Let $V_0:=V(G_{t+1}) =
\left\{v_1, \ldots, v_s \right\}$ be the
 subset of $V(G)$ remaining after this process; note that it might be the case that $V_0=\emptyset$.
 Because each $H_i$ is a bipartite, $k$-connected
subgraph of $G_i$ of  maximum size and $G$ is $10^{10}k^3 \log n$ connected,
we show that the following are true:
\begin{enumerate}[$(a)$]
\item For every $1 \leq i<j \leq t$, there are less than $4k$ independent edges between
$H_i$ and $H_j$, and
\item for every $j>i$ and $v\in V(G_j)$, the number of edges in $G$ between $v$ and $H_i$, denoted by $d_G(v,V(H_i))$, is less than
$2k$, and
\item for every $1\leq i\leq t$, there exists a set $M_i$ consisting of exactly $10^3 k^2 \log n$ independent edges, each of which has exactly one endpoint in
$H_i$.
\end{enumerate}

Indeed, for showing $(a)$, note that if there are at least $4k$
independent edges between $H_i$ to $H_j$, by pigeonhole principle,
at least $k$ of them are between the same part of $H_i$ (say $S_i$)
and the same part of $H_j$ (say $S_j$).  Therefore, the graph
obtained by joining $H_i$ to $H_j$ with this set of at least $k$
edges is a $k$-connected (by Lemma \ref{lemma:joining two k
connected graphs}), bipartite graph and is larger than $H_i$,
contrary to the maximality of $H_i$.

For showing $(b)$, note that if there are at least $2k$ between $v$
and $H_i$ then there are at least $k$ edges incident with $v$ touch
the same part of $H_i$, and let $F$ be a set of $k$ such edges.
Second, we mention that joining a vertex of degree at least $k$ to a
$k$-connected graph trivially yields a $k$-connected graph. Next,
since all the edges in $F$ are touching the same part, the graph
obtained by adding $v$ to $V(H_i)$ and $F$ to $E(H_i)$, will also be
bipartite. This contradicts the maximality of $H_i$.

For $(c)$, note first that since $H_1$ is not spanning, using $(b)$
we conclude that in the construction of the bipartite subgraphs
$H_1, \ldots, H_t$ in the process above,
$$\delta(G_2)\geq 10^{10}k^3\log n-2k\geq 8000k^2\log n.$$
Therefore, using Corollary \ref{cor:maderBipartite}, it follows that
$G_2$ contains a bipartite subgraph of size at least $10^3k^2\log n$
which is also $k$-connected.

Therefore, the process does not
terminate at this point, and $H_2$ exists (that is, $t\geq 2$). It also
follows that for each $1\leq i\leq t$ we have $|V(G)\setminus
V(H_i)|\geq 10^3k^2\log n$. Next, note that $G$ is $10^{10}k^3 \log
n$ connected, and that each $H_i$ is of size at least $10^3 k^2 \log
n$. For each $i$, consider the bipartite graph with parts $V(H_i)$ and
$V(G)\setminus V(H_i)$ and with the edge-set consisting of all the
edges of $G$ which touch both of these parts. Using K\"onig's
Theorem (see \cite{West}, p. 112), it follows that if there is no
such $M_i$ of size $10^3 k^2 \log n$, then there exists a set of
strictly fewer than $10^3 k^2 \log n$ vertices that touch all the edges in this
bipartite graph (a vertex cover).  By deleting these vertices,
one can separate what is left from $H_i$ and
its complement, contrary to the fact that $G$ is $10^{10}k^3 \log n$
connected.

In order to complete the proof, we wish to reach a contradiction by showing that one can either
merge few members of $\{H_1, \ldots, H_t\}$ with vertices of $V_0$ into a $k$-connected
component or find a $k$-connected component of size at least
$10^3 k^2 \log n$ which is contained in $V_0$. In order to do so,
we define an auxiliary digraph, using a special subgraph $G'\subseteq G$,
and use Lemmas \ref{lem:main} and \ref{merge few into on} to achieve the desired contradiction. We first
describe how to find $G'$.

First, we partition $V_0$ into two sets,
say $A$ and $B$, where
$$A = \left\{v \in V_0: d_G\left(v, \bigcup_{i=1}^t
V(H_i)\right)\geq 10^4 k^3 \log n \right\},$$ and observe that,
using $(b)$, since $A\subseteq V_0$, any vertex $a \in A$ must send
edges to more than
$$10^4 k^3 \log n/(2k)=5000k^2\log n$$ distinct elements in
$X:=\{H_1,\ldots, H_t,v_1,\ldots,v_s\}$. For each $1 \leq i \leq t$,
let $M_i$ be a set as described in $(c)$. Observe that, using $(b)$,
each such $M_i$ touches more than
$$10^3k^2\log n/(4k)=250k\log n$$ distinct elements of $X\backslash\left\{H_i \right\}$.
Let $M'_i\subseteq M_i$ be a subset of size exactly $250k\log n$
such that each pair of edges in $M'_i$ touches two distinct elements
of $X\backslash\left\{H_i \right\}$, which of course are distinct from $G_i$.
Recall that $H_i=(S_i\cup T_i,E_i)$ for every $1 \leq i \leq t$.

For $Y:=\{S_1,\ldots, S_t, T_1,\ldots,T_t,v_1,\ldots,v_s\}$,
let $$\Phi:Y\rightarrow \{L,R\}$$ be a mapping, generated according to the following random process:

Let $X_1,\ldots,X_t,Y_1,\ldots,Y_s \sim \text{Bernoulli}(1/2)$ be
mutually independent random variables. For each $1\leq i\leq t$, if
$X_i=1$, then let $\Phi(S_i)=L$ and $\Phi(T_i)=R$. Otherwise, let
$\Phi(S_i)=R$ and $\Phi(T_i)=L$. For every $1\leq j\leq s$, if
$Y_j=1$, then let $\Phi(v_j)=L$, and otherwise $\Phi(v_j)=R$. Now,
delete all of the edges between two distinct elements of $Y$ which receive
the same label according to $\Phi$.

Finally, define $G'$ as the
spanning bipartite graph of $G$ obtained by deleting all of the edges
within $A$ and for distinct $i$ and $j$, the edges between $H_i$ and $H_j$ which are not contained in $M'_i\cup M_j'$.  

Recall by construction, using $\Phi$ we generated labels at random;
therefore, by using Chernoff bounds (for instance see \cite{AS}), one can easily check that
with high probability the following hold:
\begin{enumerate}[$(i)$]
\item For every $1\leq i\leq t$, each set $M'_i\cap E(G')$ touches at least
(say) $120 k \log n$ other elements of $X$, and
\item for each $b\in B$, the degree of $b$ into $A\cup B$ is at
least (say) $d_{G'}(b,A\cup B)\geq 10^5k^3\log n$, and
\item for each vertex $a\in A$, there exist edges between $a$ and $\cup_{i=1}^t V(H_i)$ that touch at least (say) $2000k^2\log n$
distinct members of $\{H_1, \ldots, H_t\}$.
\end{enumerate}

Note that here we relied on the luxury of losing the $\log n$ factor for using
Chernoff bounds, but it seems like we could easily handle this ``cleaning
process'' completely by hand.

Now we are ready to define our auxiliary digraph $D$. To this end,
we first orient edges (again, bidirectional edges are allowed, and
un-oriented edges are considered as bidirectional) of $G'$ in the
following way:

For every
$1\leq i\leq t$, we orient all of the edges in $E(G')\cap M'_i$ out of
$H_i$. We orient all of the edges between $A$ and $\cup_{i=1}^t V(H_i)$
out of $A$. We orient edges between $B$ and $\cup_{i=1}^t V(H_i)$
arbitrarily, and we orient the remaining edges within $A\cup B$ in
both directions.

Now, we define $D$ to be the
digraph with vertex set $V(D)=X$, and $\overrightarrow{xy}\in E(D)$ if and only
if there exists an edge between $x$ and $y$ in $G'$ which is
oriented from $x$ to $y$.

In order to complete the proof, we first note that with high probability $D$ is a digraph on at most
$n$ vertices with out-degree $\delta^+(D)>(k-1)\lceil\log n\rceil$.
This follows
immediately from Properties $(i)$-$(iii)$ as well as the way we oriented
the edges. Therefore, one can apply Lemma \ref{lem:main} to find a
sub-digraph $D'\subseteq D$ such that
\begin{enumerate}
\item For all $v \in V(D')$ we have $d_{D'}^{+}(v) \geq d_{D}^{+}(v) - (k-1) \left\lceil \log n\right\rceil$, and
\item $\kappa(U(D')) \geq k$.
\end{enumerate}

In fact, with high probability, $\delta^+(D)\geq 120k \log n \geq k + (k-1) \left\lceil \log n\right\rceil.$
Note that by construction, every pair of edges which are oriented
out of some $H_i$ must be independent and go to different components.
Using Property $1.$ above combined with the fact that $\delta^+(D') \geq \delta^+(D) - (k-1)
\left\lceil \log n\right\rceil\geq k$, we may conclude that the subgraph
$G''\subseteq G'$ induced by the union of all the components in
$V(D')$ satisfies $G''\in \mathcal F_{U(D')}(V(D'))$. Applying
Lemma \ref{merge few into on} with $X = V(D')$ and $R = U(D')$,
it follows that $\kappa(G'')\geq k$.

In order to obtain the desired
contradiction, we consider the following two cases:

\textbf{Case 1:} $V(G'')$ contains $V(H_i)$ for some $i$. We note
that this case is actually impossible because it would contradict
the maximality of $H_i$ for the minimal index $i$ such that $V(H_i)\subseteq
V(G'')$.

\textbf{Case 2:} $V(G'')\subseteq A\cup B$. We note that in this
case, there must be at least one vertex $b\in B\cap V(G'')$. Indeed,
$G''$ is $k$-connected, and there are no edges within $A$. Now, it
follows from Properties $1.$\ and $(ii)$ above that
$$d_{D'}^+(b)\geq d_{D}^+(b)-(k-1)\lceil \log n\rceil \geq
10^4k^3\log n.$$ Thus, it follows that $|V(G'')|\geq 10^4k^3\log n$.
Combining this observation with the facts that $G''$ is
$k$-connected and $V(G'')\subseteq A\cup B$, we obtain a
contradiction.  This case can not arise because $G''$ should have
been included as one of the bipartite subgraphs $\{H_1, \ldots,
H_t\}$.

This completes the proof.
\end{proof}

{\bf Acknowledgments.} The authors would like to thank the anonymous
referees for valuable comments. The second author would also like to
thank Andrzej Grzesik, Hong Liu and Cory Palmer for fruitful
discussions in a previous attempt to attack this problem.


\end{document}